\documentclass[a4paper,english,11pt]{amsart}
\usepackage{babel}
\usepackage{color}
\usepackage{graphics}
\usepackage{graphicx}
\usepackage{amsmath,sfheaders}
\usepackage[latin1]{inputenc}
\usepackage[T1]{fontenc}
\usepackage{textcomp}
\usepackage{stmaryrd}
\usepackage{cases}
\usepackage{pgf}
\usepackage[all]{xy}
\usepackage{pgf,amssymb,amsmath}
\usepackage{amssymb,amsmath,amstext,stmaryrd,euscript,tabularx,epsfig}
\usepackage{geometry}
\usepackage{fancyhdr}
\usepackage{pstricks}

\newcommand{\supp}{\textup{supp}}

\newcommand{\D}{\mathbb{D}}

\newcommand{\C}{\mathbb{C}}
\newcommand{\p}{\mathbb{P}}
\newcommand{\pe}{\textup{ := }}

\newcommand{\rat}{\textup{Rat}}
\newcommand{\Per}{\textup{Per}}
\newcommand{\bif}{\textup{bif}}

\newcommand{\Mand}{\textup{\textbf{M}}}

\begin{document}

\def\theequation{\thesection.\arabic{equation}}
\def\sqw{\hbox{\rlap{\leavevmode\raise.3ex\hbox{$\sqcap$}}$%
\sqcup$}}
\def\sqb{\hbox{\hskip5pt\vrule width4pt height4pt depth1.5pt%
\hskip1pt}}

\newtheorem{defi}{Definition}[section]
\newtheorem{tm}{Theorem}[section]
\newtheorem{prop}[tm]{Proposition}
\newtheorem{nota}{Notation}[section]
\newtheorem{rem}{Remark}[section]
\newtheorem{lm}[tm]{Lemma}
\newtheorem{cor}[tm]{Corollary}
\newtheorem{ctrex}[rem]{Counter-example}
\newtheorem{ex}[rem]{Example}
\newtheorem{apl}[tm]{Application}
\newtheorem{fait}{Fact}
\newtheorem*{TM}{Main Theorem}

\title{On the geometry of bifurcation currents for quadratic rational maps.}
\author{Fran\c{c}ois Berteloot and Thomas Gauthier}

\maketitle
\begin{abstract}
We describe the behaviour at infinity of the bifurcation current in the moduli space of quadratic rational maps.
To this purpose, we extend  it to some closed, positive $(1,1)$-current on a two-dimensional complex projective space and then
compute the Lelong numbers and the self-intersection of the extended current.
\end{abstract}
\section{Introduction.}
For any holomorphic family $\left(f_{\lambda}\right)_{\lambda\in M}$ of degree $d$ rational maps on $\p^1$, the bifurcation locus is the subset of the parameter space $M$
where the Julia sets $J_{\lambda}$ of $f_{\lambda}$ does not move continuously with $\lambda$. In their seminal paper \cite{MSS}, Ma\~{n}\'e, Sad and Sullivan have shown that the bifurcation locus is nowhere dense in $M$ and coincides with the closure of the set of parameters for which $f_{\lambda}$ admits a non-persistent neutral cycle.\\
For the quadratic polynomial family $\left(z^2+\lambda\right)_{\lambda\in \C}$,  the bifurcation locus is the boundary of the Mandelbrot set and, in particular,  is bounded.
The situation is much more complicated for quadratic rational maps. Their moduli space $\mathcal{M}_2$ can be identified with the complement in a complex projective space $\p^2$ of a line at infinity
$\mathbb{L}_\infty$ around which the behaviour of the bifurcation locus is far to be completely understood.\\
The first results in this direction are due to J. Milnor \cite{Milnor3} who studied the curves
\begin{center}
$\Per_n(w)\pe\{[f]\in\mathcal{M}_2$ s.t. $f$ has a $n$-cycle of multiplier $w\}$
\end{center}
in the projective space $\p^2$ (see Proposition \ref{propMil}). Sharper results have then been obtained by A. Epstein \cite{epstein} and yields a precise description of the intersections of 
$\Per_n(w)$ with $\mathbb{L}_\infty$ (see Proposition \ref{propEps}). It should be stressed that the deformations of the Mandelbrot set, provided by intersecting the bifurcation locus with the \emph{lines} $Per_1(t)$ for $\vert t\vert <1$ (see Proposition \ref{propMouvHol}), play a central role in these investigations. In particular, the work of C. Petersen
\cite{petersen} on the collapsing of limbs (see the beginning of the fourth section) is crucial in Epstein's proof.\\

Currents not only provide an appropriate framework for studying bifurcations from a measure-theoretic viewpoint, but are also very well suited to investigate
the asymptotic distribution of the hypersurfaces $\Per_n(w)$. Let us recall that L. DeMarco has proved that the bifurcation locus of any
 holomorphic family supports  a closed, positive $(1,1)$-current \cite{DeMarco1}. This current is denoted  $T_\bif$ and called \emph{bifurcation current}. We refer the reader to the survey \cite{DuS} or the lecture notes \cite{Ber} for a report on recent results involving bifurcation currents and further references.\\
The link between the bifurcation current $T_\bif$ and the hypersurfaces $\Per_n(w)$ relies on the fact that the Lyapunov exponent $L(\lambda)$ of $f_\lambda$
(with respect to its maximal entropy measure) is a global potential for $T_\bif$ (see \cite{DeMarco2} or \cite{BB1}). G. Bassanelli and the first author have indeed deduced from this property
the convergence of the weighted integration currents on $\Per_n(w)$
\begin{center}
$\lim_n \frac{1}{d^n}[\Per_n(w)]= T_\bif$
\end{center}
 for $\vert w\vert <1$ or for any $w \in \C$ when the hyperbolic parameters are dense in $M$ (see \cite{BB3}).\\
 
 The present paper is devoted to the study of the behaviour at infinity of the bifurcation current in the moduli space $\mathcal{M}_2$ of quadratic rational maps. We first show that this current extends naturally to a closed, positive $(1,1)$-current $\hat T_\bif$ on $\p^2$. We relate this current with the hypersurfaces $\Per_n(w)$ and precise its support
 (see Theorem \ref{teoExt}). We then use Epstein's result to compute the Lelong numbers of  $\hat T_\bif$ at any point of the line at infinity $\mathbb{L}_\infty$ and get a precise
 description of the measure $\hat T_\bif\wedge [\mathbb{L}_\infty]$ (see Theorem \ref{teoLelN}). We finally describe the support of the so-called bifurcation measure $ T_\bif\wedge  T_\bif$ and compare this measure  with $\hat T_\bif\wedge \hat T_\bif$ (see Propositions \ref{propSupM} and \ref{propComM}).\\
 
 \noindent{\bf Notations.} The complex plane will be denoted $\C$ and the euclidean unit disc in $\C$ will be denoted $\D$. The $k$-dimensional complex projective space is denoted
 $\p^k$. A ball of radius $r$ centered at $x$  
 in some metric space is denoted $B(x,r)$.
 
\section{Preliminaries.}

\subsection{Lelong numbers and currents in $\p^2$.}

\par Let $T$ be a closed, positive, $(1,1)$-current  in $\p^2$. The \emph{mass} of $T$ on an open set $U\subset\p^2$ is given by
\begin{center}
$\|T\|_U\pe\displaystyle\int_UT\wedge\omega$
\end{center}
where $\omega$ is the Fubini-Study form on $\p^2$. When $V$ is an algebraic curve in $\p^2$ and $[V]$ is the integration current on $V$, then $\|[V]\|_{\p^2}=\deg(V)$.
\par For any $x\in\p^2$, the \emph{Lelong number} of $T$ at $x$ is given by
\begin{center}
$\displaystyle\nu(T,x)\pe\lim_{r\rightarrow0}\frac{1}{r^2}\|T\|_{B(x,r)}$.
\end{center}
These numbers somehow measure the singularities of $T$. If $V$ is an algebraic curve in $\p^2$, then $\nu([V],x)$ is the multiplicity of $V$ at $x$. We will mainly use the two following properties of Lelong numbers (see \cite{Demailly} Proposition 5.12 page 160 and Corollary 7.9 page 169). 

\begin{tm}[Demailly]
\begin{enumerate}
\item If $T_n\longrightarrow T$, then for any $x\in\p^2$,
\begin{eqnarray*}
\nu(T,x)\geq\limsup\limits_{n\rightarrow+\infty}\nu(T_n,x).
\end{eqnarray*}
\item If $T_1$ and $T_2$ are such that $T_1\wedge T_2$ is well-defined, then for any $x\in\p^2$,
\begin{eqnarray*}
\nu(T_1\wedge T_2,x)\geq\nu(T_1,x)\cdot\nu(T_2,x).
\end{eqnarray*}
\end{enumerate}
\label{semicont}
\end{tm}

\subsection{The moduli space $\mathcal{M}_2$ of quadratic rational maps.}\label{sectionm2}

\par The space $\rat_2$ of quadratic rational maps on $\p^1$ may be viewed as a Zarisky-open subset of $\p^5$ on which the group of M\"obius transformations acts by conjugation. The \emph{moduli space} $\mathcal{M}_2$ is, by definition, the quotient resulting from this action. Denote by $\alpha,\beta,\gamma$ the multipliers of the three fixed points of $f\in\rat_2$ and $\sigma_1=\alpha+\beta+\gamma$, $\sigma_2=\alpha\beta+\alpha\gamma+\beta\gamma$ and $\sigma_3=\alpha\beta\gamma$ the symmetric functions of these multipliers. Milnor has proved that $(\sigma_1,\sigma_2)$ is a good parametrization of $\mathcal{M}_2$ (see \cite{Milnor3}).

\begin{tm}[Milnor]
The map $\left[f\right]\in\mathcal{M}_2 \longmapsto (\sigma_1,\sigma_2)\in\C^2$ is a canonical biholomorphism.
\label{coordonneesm2}
\end{tm}

\par To study the curves
\begin{center}
$\Per_n(w)\pe\{[f]\in\mathcal{M}_2$ s.t. $f$ has a $n$-cycle of multiplier $w\}$,
\end{center}
 it is useful to compactify $\mathcal{M}_2$. In this context, it turns out that the projective compactification
\begin{center}
 $  \xymatrix {\relax
[f]\in\mathcal{M}_2 \ar@{^{(}->}[r] & [\sigma_1:\sigma_2:1]\in\p^2}$
\end{center}
is suitable. Denote by $\mathbb{L}_\infty$ the line at infinity of $\mathcal{M}_2$, i.e. 
\begin{center}
$\mathbb{L}_\infty=\{[\sigma_1:\sigma_2:0]\in\p^2$ / $(\sigma_1,\sigma_2)\in\C^2\setminus\{0\}\}$.
\end{center}
 Milnor has described the behavior of $\Per_n(w)$ at infinity as follows (see \cite{Milnor3} Lemmas 3.4 and section 4).

\begin{prop}[Milnor]\label{propMil}
\begin{enumerate}
\item For any $w\in\C$ the curve $\Per_1(w)$ is a line in $\p^2$ whose equation in $\C^2$ is $(w^2+1)\lambda_1-w\lambda_2-(w^3+2)=0$ and which intersects the line at infinity $\mathbb{L}_\infty$ at the point $[w:w^2+1:0]$.
\item For any $n\geq2$ and any $w\in\C$, the curve $\Per_n(w)$ is an algebraic curve in $\p^2$ whose degree equals the number $d(n)\sim 2^{n-1}$ of $n$-hyperbolic components of the Mandelbrot set. In addition, the intersection $\Per_n(w)\cap\mathbb{L}_\infty$is contained in the set $\{[1:u+1/u:0]\in\p^2$ / $u^q=1$ with $q\leq n\}$.
\end{enumerate}
\label{propMilnor}
\end{prop}

\par For reasons which will appear to be clear later, we shall denote by $\mathbb{L}_{\infty,\bif}$ the subset of $\mathbb{L}_\infty$ defined by
\begin{center}
$\mathbb{L}_{\infty,\bif}:=\{[1:e^{i\theta}+1/e^{i\theta}:0]\in\p^2 / \theta\in [0,2\pi]\}$.
\end{center}

\par Golberg and Keen showed how the Mandelbrot set $\Mand$ determines the connectedness locus for quadratic rational maps having an attracting fixed point (see \cite{goldbergkeen} Section $1$ and \cite{BB3} Theorem $5.4$ for a potential theoretic approach).

\begin{tm}[Goldberg-Keen, Bassanelli-Berteloot]\label{propMouvHol}
There exists a holomorphic motion $\sigma:\D\times\Per_1(0)\longrightarrow\C^2$ such that $\Mand_t\pe\sigma_t(\Mand)\Subset\Per_1(t)\cap\C^2$ is the connectedness locus of the line $\Per_1(t)$. Moreover, if $|w|\leq1$ and $n\geq1$, then $\Per_1(t)\cap\Per_n(w)\subset\Mand_t$.
\label{tmmvtmand}
\end{tm}

\subsection{Bifurcation currents in $\mathcal{M}_2$.}
Every rational map $f$ of degree $d\geq2$ on the Riemann sphere admits a maximal entropy measure $\mu_f$. The Lyapunov exponent of $f$ with respect to this measure is defined by
\begin{center}
$L(f)=\int_{\p^1}\log|f'|\mu_f$.
\end{center}
It turns out that the function $L:\rat_2\longrightarrow L(f)$ is \emph{p.s.h} and continuous on $\rat_2$. We still denote by $L$ the function induced on $\mathcal{M}_2$. The \emph{bifurcation current} $T_\bif$  is a closed positive (1, 1)-current on $\mathcal{M}_2$ which may be defined by
\begin{center}
$T_\bif=dd^cL$.
\end{center}
As it has been shown by DeMarco \cite{DeMarco2}, the support of $T_\bif$ concides with the bifurcation locus of the family in the classical sense of Ma\~ne-Sad-Sullivan (see also \cite{BB1}, Theorem 5.2). Using properties of the Lyapounov exponent, Bassanelli and the first author proved that the curves $\Per_n(0)$ equidistribute the bifurcation current (see \cite{BB3}).

\begin{tm}[Bassanelli-Berteloot]
The sequence $2^{-n}[\Per_n(0)]$ converges to $T_\bif$ in the sense of currents in $\mathcal{M}_2$. Moreover, $2^{-n}[\Per_n(0)]|_{\Per_1(0)}$ converges weakly to $\Delta(L|_{\Per_1(0)})=\frac{1}{2}\mu_\Mand$, where $\mu_\Mand$ is the harmonic measure of $\Mand$.
\label{tmequidistribution}
\end{tm}
\par As the function $L$ is continuous, one can define the \emph{bifurcation measure} $\mu_\bif$ of the moduli space $\mathcal{M}_2$ as the Monge-Amp\`ere mass of the  function $L$, i.e.
\begin{center}
$\mu_\bif\pe(dd^cL)^2\pe dd^c\left(L dd^cL\right)$.
\end{center}
This measure has been introduced by Bassanelli and the first author \cite{BB1}. Buff and Epstein also studied it in \cite{buffepstein}. Recall that a rational map $f$ is said to be strictly poscritically finite if each critical point of $f$ is preperiodic but not periodic. We denote by $\mathcal{SPCF}$ the set of classes of quadratic strictly postcritically finite rational maps. We also denote by $\mathcal{S}$ the set of \emph{Shishikura} maps, i.e. $\mathcal{S}\pe\{[f_0]\in\mathcal{M}_2$ / $f_0$ has $2$ distinct neutral cycles$\}$. Combining the work of Bassanelli and the first author with that of Buff and Epstein, we have the following.

\begin{tm}[Bassanelli-Berteloot, Buff-Epstein]
\hspace{1cm}
\begin{center}
$\supp(\mu_\bif)=\overline{\mathcal{SPCF}}=\overline{\mathcal{S}}$.
\end{center}
\label{tmsuppmubif}
\end{tm}

These results are still valid in moduli spaces of degree $d$ rational maps for any $d\geq2$. Notice that Buff and the second author \cite{Article2} have proved that flexible Latt\`es maps belong to $\supp(\mu_\bif)$.

\section{Extension of the bifurcation current to Milnor's compactification.}

\par As the current $T_\bif$ is a weak limit of weighted integration currents on curves which are actually defined on $\p^2$, one may expect to naturally extend it to $\p^2$. We will show how this is possible and prove some basic properties of the extended current. More precisely, we establish the following result which may be considered as a measure-theoretic counterpart of Milnor's description of the bifurcation locus in $\mathcal{M}_2$.

\begin{tm}\label{teoExt}
There exists a closed positive $(1,1)$-current $\hat T_\bif$ on $\p^2$ whose mass equals $1/2$ and such that:
\begin{enumerate}
\item $\hat T_\bif|_{\C^2}=T_\bif$,
\item $2^{-n}[\Per_n(0)]$ converges to $\hat T_\bif$ in the sense of currents on $\p^2$,
\item $\supp(\hat T_\bif)=\supp(T_\bif)\cup\mathbb{L}_{\infty,\bif}$.
\end{enumerate}
\label{tmextension}
\end{tm}

\par We shall use the two following lemmas which are of independant interest. The first one is essentially due to Milnor (see Theorem 4.2 in \cite{Milnor3} or Theorem 2.3 in \cite{BB3}). The second one will be proved at the end of the section.

\begin{lm}
$\|[\Per_n(0)]\|_{\p^2}\sim 2^{n-1}$.
\label{lmmass}
\end{lm}		

\begin{lm}
$\displaystyle\overline{\bigcup_{n\geq1}\Per_n(0)}\cap\mathbb{L}_\infty=\mathbb{L}_{\infty,\bif}$.
\label{lmcluster}
\end{lm}

\begin{proof} We first justify the existence of $\hat T_\bif$. According to the Skoda-El Mir Theorem (see \cite{Demailly} Theorem 2.3 page 139), the trivial extension of $T_\bif$ through the line at infinity $\mathbb{L}_\infty$ is a closed positive $(1,1)$-current on $\p^2$ if $T_\bif$ has locally bounded mass near $\mathbb{L}_\infty$. The Lemma \ref{lmmass}, combined with the fact that $2^{-n}[\Per_n(0)]$ converges to $T_\bif$ on $\C^2$ (see Theorem \ref{tmequidistribution}), immediatly yields $\|T_\bif\|_{\C^2}\leq 1/2$. The extension $\hat T_\bif$ of $T_\bif$ therefore exists and $\|\hat T_\bif\|_{\p^2}\leq1/2$.

~

\par Let us now prove that $2^{-n}[\Per_n(0)]$ converges to $\hat T_\bif$ on $\p^2$ and $\|\hat T_\bif\|_{\p^2}=1/2$. By Lemma \ref{lmmass}, $(2^{-n}[\Per_n(0)])_{n\geq1}$ is a sequence of closed positive $(1,1)$-currents with uniformly bounded mass on $\p^2$. According to the compactness porperties of such families of currents, it suffices to show that $\hat T_\bif$ is the only limit value of $2^{-n}[\Per_n(0)]$.
\par Assume that $2^{-n_k}[\Per_{n_k}(0)]$ converges to $T$ on $\p^2$. By Siu's Theorem (see \cite{Demailly} Theorem 8.16 page 181), one has $T=S+\alpha[\mathbb{L}_\infty]$, where $S$ has no mass on $\mathbb{L}_\infty$. But, by Lemma \ref{lmcluster}, the line $\mathbb{L}_\infty$ is not contained in $\supp(T)$. Thus $\alpha=0$ and $T$ has no mass on $\mathbb{L}_\infty$. Since $T|_{\C^2}=\lim_k2^{-n_k}[\Per_{n_k}(0)]=T_\bif$, this shows that $T$ is actually the trivial extension of $T_\bif$ through $\mathbb{L}_\infty$ and therefore, according to the first part of the proof, equals $\hat T_\bif$.
\par Now, Lemma \ref{lmmass} immediatly yields $\|\hat T_\bif\|_{\p^2}=1/2$.
\end{proof}

\par \textit{Proof of Lemma \ref{lmcluster}.} Let $\sigma:\Per_1(0)\times\D\longrightarrow \bigcup_{|u|<1}\Per_1(u)$ be the holomorphic motion given by Theorem \ref{tmmvtmand}. Let us set $\Mand_u\pe\sigma_u(\Mand)
$. As $\Mand$ is compact and $\sigma$ is continuous one has
\begin{center}
$\displaystyle\bigcup_{|u|\leq r<1}\Mand_u\Subset\C^2$ for all $0<r<1$.
\end{center}
Let us also recall that $\Per_n(0)\cap \Per_1(u)\subset\Mand_u$ for $n\geq 2$ and $u\in\D$.
\par We now proceed by contradiction and assume that there exists
\begin{center}
$\zeta\in\big(\overline{\bigcup_n\Per_n(0)}\cap\mathbb{L}_\infty\big)\setminus\mathbb{L}_{\infty,\bif}$.
\end{center}
By Proposition \ref{propMilnor}, $\zeta\in\Per_1(u_0)$ for some $u_0\in\D$. Let us pick $\lambda_k\in\Per_{n_k}(0)$ such that $\lambda_k\rightarrow\zeta$. Then there exists $u_k\in\D$ such that $u_k\rightarrow u_0$ and $\lambda_k\in\Per_1(u_k)$, which is impossible since
\begin{center}
$\lambda_k\in\Per_{n_k}(0)\cap\Per_1(u_k)\subset\Mand_{u_k}\subset\displaystyle\bigcup_{|u|<r}\Mand_u$
\end{center}
for some $|u_0|<r<1$.\hfill$\Box$

\section{Lelong numbers of the bifurcation current at infinity.}

\par The aim of the present section is to compute the Lelong numbers of $\hat T_\bif$ at any point. This is related to previous works of Petersen (\cite{petersen}) and Epstein (\cite{epstein}) which we briefly describe.

\par Let $\heartsuit$ be the main cardioid of the Mandelbrot set and $\mathcal{L}_{p/q}$ the $p/q$-\emph{limb} of $\Mand$ (see \cite{branner} page $84$). Denote by $d_{p/q}(n)$ the number of $n$-hyperbolic component of $\mathcal{L}_{p/q}$ and set
\begin{center}
$D_{p/q}(n)=\left\{\begin{array}{ll}
d_{p/q}(n) & \text{if }p/q=1/2,\\
2d_{p/q}(n) & \text{otherwise.}
\end{array}\right.$
\end{center}
Let $\sigma$ be the holomorphic motion of $\Per_1(0)$ given by Theorem \ref{tmmvtmand} and
\begin{center}
$\infty_{p/q}\pe[1:2\cos(2\pi p/q):0]$
\end{center}
if $1\leq p\leq q/2$ and $p\wedge q=1$. Petersen has proved that the limb $\sigma_u(\mathcal{L}_{p/q})$ of $\Mand_u$ disappears when $u$ tends non-tangentially to $e^{-2i\pi p/q}$. Using this result, Epstein has precisely described the intersection $\Per_n(w)\cap\mathbb{L}_\infty$. He namely proved the following:

\begin{prop}[Epstein]\label{propEps}
For any $w\in\C$ and any $n\geq2$, 
\begin{center}
$\displaystyle[\Per_n(w)]\wedge[\mathbb{L}_\infty]=\sum_{1\leq p\leq q/2\leq n/2 \atop p\wedge q=1}\nu([\Per_n(w)],\infty_{p/q})\delta_{\infty_{p/q}}=\sum_{1\leq p\leq q/2\leq n/2 \atop p\wedge q=1}D_{p/q}(n)\delta_{\infty_{p/q}}$.
\end{center}
\label{propEpstein}
\end{prop}

\par From the above Proposition and by using the previous section, we deduce the following.

\begin{tm}\label{teoLelN}
\begin{enumerate}
\item The Lelong numbers of $\hat T_\bif$ are given by 
\begin{center}
$\nu(\hat T_\bif,a)=\left\{\begin{array}{cl}
1/6 & \text{ if }a=\infty_{1/2},\\
1/(2^q-1) & \text{ if }a=\infty_{p/q} \text{ and },q\geq3\\
0 & \text{ if } a\in\p^2\setminus\{\infty_{p/q}\ / \ p\wedge q=1,\ 1\leq p\leq q/2\}.
\end{array}\right.$
\end{center}
\item The measure $\hat T_\bif\wedge [\mathbb{L}_\infty]$ is discrete and given by
\begin{center}
$\hat T_\bif\wedge[\mathbb{L}_\infty]=\displaystyle\sum_{1\leq p\leq q/2 \atop p\wedge q=1}\nu(\hat T_\bif,\infty_{p/q})\delta_{\infty_{p/q}}$.
\end{center}
\end{enumerate}
\label{tmnombrelelong}
\end{tm}

The proof requires the two following lemmas. The first one is a consequence of Theorem \ref{tmextension} and the second one relies on a simple computation. They will be proved at the end of the section.

\begin{lm}
The measure $\hat T_\bif\wedge[\mathbb{L}_\infty]$ is a well-defined positive measure on $\p^2$ of mass $1/2$.
\label{lmintersect}
\end{lm}

\begin{lm}
Let $\mu\pe\displaystyle\frac{1}{6}\delta_{\infty_{1/2}}+\sum_{1\leq p\leq q/2 \atop p\wedge q=1, \ q\geq3}\displaystyle\frac{1}{2^q-1}\delta_{\infty_{p/q}}$, then $\mu$ has mass $1/2$.
\label{lmmuinfty}
\end{lm}

\begin{proof}
First observe that $\nu(\hat T_\bif,a)=0$ when $a\notin \mathbb{L}_{\infty,\bif}$, since then $\hat T_\bif$ has a continuous potential in a neighborhood of $a$. Let us now pick  $1\leq p\leq q/2$ such that $p\wedge q=1$. By item $(3)$ of Theorem \ref{tmextension} and Theorem \ref{semicont}, we have
\begin{center}
$\nu(\hat T_\bif,\infty_{p/q})\geq\displaystyle\limsup_{n\rightarrow\infty}\nu(2^{-n}[\Per_n(0)],\infty_{p/q})$.
\end{center}
Proposition \ref{propEpstein} states that $\nu([\Per_n(0)],\infty_{p/q})=D_{p/q}(n)$ is the number of $n$-hyperbolic components of the union $\mathcal{L}_{p/q}\cup\mathcal{L}_{-p/q}$ of limbs of the Mandelbrot set. Thus, by Theorem \ref{tmequidistribution},
\begin{center}
$2^{-n}\nu([\Per_n(0)],\infty_{p/q})\longrightarrow \frac{1}{2}\mu_\Mand(\mathcal{L}_{p/q}\cup\mathcal{L}_{-p/q})$.
\end{center}
On the other hand, Bullett and Sentenac have proved that $\mu_\Mand(\mathcal{L}_{p/q})=\mu_\Mand(\mathcal{L}_{-p/q})=\frac{1}{2^q-1}$ (see \cite{bullett}). This gives $\nu(\hat T_\bif,\infty_{p/q})\geq \frac{1}{2^q-1}$ if $p/q\neq1/2$ and $\nu(\hat T_\bif,\infty_{1/2})\geq \frac{1}{2(2^2-1)}=\frac{1}{6}$.
\par Let $a\in\mathbb{L}_\infty$. By Theorem \ref{semicont} , we have $\nu(\hat T_\bif\wedge [\mathbb{L}_\infty],a)\geq\nu(\hat T_\bif,a)\nu([\mathbb{L}_\infty],a)$. As $\mathbb{L}_\infty$ is a line, we get $\nu(\hat T_\bif\wedge [\mathbb{L}_\infty],a)\geq\nu(\hat T_\bif,a)$. For $a=\infty_{p/q}$, we thus have $\nu(\hat T_\bif\wedge [\mathbb{L}_\infty],\infty_{p/q})\geq \frac{1}{2^q-1}$ if $q\geq3$ and $\nu(\hat T_\bif\wedge [\mathbb{L}_\infty],\infty_{1/2})\geq\frac{1}{6}$, which we restate as
\begin{center}
$\hat T_\bif\wedge [\mathbb{L}_\infty]\geq\mu$.
\end{center}
Since, according to Lemmas \ref{lmintersect} and \ref{lmmuinfty}, the  positive measures $\hat T_\bif\wedge[\mathbb{L}_\infty]$ and $\mu$ have the same mass, this yields $\hat T_\bif\wedge[\mathbb{L}_\infty]=\mu$. We thus get point $(2)$ and $\nu(\hat T_\bif,\infty_{p/q})=\frac{1}{2^q-1}$ for $q\geq3$ and $\nu(\hat T_\bif,\infty_{1/2})=\frac{1}{6}$.
\par From $\hat T_\bif\wedge [\mathbb{L}_\infty]\geq\nu(\hat T_\bif,a)\delta_a$ for any $a\in\mathbb{L}_\infty$ and $\hat T_\bif\wedge [\mathbb{L}_\infty]=\mu$ we get $\nu(\hat T_\bif,a)=0$ for $a\neq\infty_{p/q}$.
\end{proof}

\par \textit{Proof of Lemma \ref{lmintersect}. } Let us first remark that, by Theorem \ref{tmextension}, $\supp(\hat T_\bif)\cap\mathbb{L}_\infty=\mathbb{L}_{\infty,\bif}$ and let us decompose $\p^2$ as the disjoint union of the line $\Per_1(0)$ and (a copy of) $\C^2$. Let us stress that with these notations one has $\mathbb{L}_\infty\setminus\Per_1(0)\cap\mathbb{L}_\infty\subset \C^2$.
By Proposition \ref{propMilnor}, the set $\mathbb{L}_{\infty,\bif}$ is compact in $\C^2$. By definition, the current $\hat T_\bif$ has a potential $u$ in $\C^2$ which is continuous in $\C^2\setminus\mathbb{L}_\infty$. By item $(3)$ of Theorem \ref{tmextension}, this potential is actually continuous in $\C^2\setminus\mathbb{L}_{\infty,\bif}$. Let $B_1$ be a ball of $\C^2$ containing $\mathbb{L}_{\infty,\bif}$ and $(B_i)_{i\geq2}$ be a covering of $\C^2\setminus B_1$ by balls such that $\overline{B_i}\cap \mathbb{L}_{\infty,\bif}=\emptyset$ for all $i\geq2$.
\par As $\{u\text{ is unbounded}\}\subset \mathbb{L}_{\infty,\bif}$ and $\mathbb{L}_{\infty,\bif}\cap \partial B_i=\emptyset$ and $B_i$ is pseudoconvex for any $i\geq1$, a result of Demailly asserts that $\hat T_\bif|_{\C^2}\wedge[\mathbb{L}_\infty]=dd^cu\wedge [\mathbb{L}_\infty]$ is well-defined (see \cite{Demailly} Proposition 4.1 page 150). Since, by Proposition \ref{propMilnor} and Theorem \ref{tmnombrelelong}, $\mathbb{L}_\infty$ and $\supp(\hat T_\bif)$ don't intersect in a neighborhood of $\Per_1(0)$, the measure $\hat T_\bif\wedge[\mathbb{L}_\infty]$ is a well-defined positive measure on $\p^2$.
\par Finally, B\'ezout Theorem asserts that $\|\hat T_\bif\wedge[\mathbb{L}_\infty]\|_{\p^2}=\|\hat T_\bif\|_{\p^2}\cdot\|[\mathbb{L}_\infty]\|_{\p^2}=1/2$.\hfill$\Box$

~

\par \textit{Proof of Lemma \ref{lmmuinfty}. } Denote by $\phi(n)$ the Euler function. As the sets $\{p \text{ s.t. } 1\leq p\leq q/2,\ p\wedge q=1\}$ and $\{q-p \text{ s.t. } 1\leq p\leq q/2,\ p\wedge q=1\}$ have same cardinality, we get
\begin{eqnarray*}
\mu(\p^2) & = & \frac{1}{6}+\sum_{1\leq p\leq q/2,\atop p\wedge q=1,\ q\geq3}\frac{1}{2^q-1}=\frac{1}{2(2^2-1)}+\frac{1}{2}\sum_{1\leq p< q,\atop p\wedge q=1, \ q\geq3}\frac{1}{2^q-1}\\
& = & \frac{1}{2}\sum_{q\geq2}\left(\sum_{1\leq p< q,\atop p\wedge q=1}\frac{1}{2^q-1}\right)=\frac{1}{2}\sum_{q\geq2}\frac{\phi(q)}{2^q-1}.
\end{eqnarray*}
A classical result (see \cite{hardywright} Theorem 309 page 258) asserts that the series $\sum_{n\geq1}\phi(n)\frac{x^n}{1-x^n}$ locally uniformly converges on $\D$ and that its sum is $\frac{x}{(1-x)^2}$. Therefore,
\begin{eqnarray*}
\mu(\p^2)=\frac{1}{2}\left(\sum_{q\geq1}\frac{\phi(q)}{2^q-1}-\phi(1)\right)=\frac{1}{2}\left(\sum_{q\geq1}\frac{\phi(q)(\frac{1}{2})^q}{1-(\frac{1}{2})^q}-\phi(1)\right)=\frac{1}{2}.
\end{eqnarray*}
\hfill$\Box$

\section{Behavior of the bifurcation measure near infinity.}

\par One often compares the moduli space $\mathcal{M}_2$ of quadratic rational maps with the moduli space $\mathcal{P}_3$ of cubic polynomials. In this section, we enlight some important differences between these two spaces. We first show that the bifurcation measure is not compactly supported in $\mathcal{M}_2$.

\begin{prop}\label{propSupM}
The cluster set of the support of $\mu_\bif$ in $\p^2$ is precisely $\mathbb{L}_{\infty,\bif}$.
\label{propunbouned}
\end{prop}

\begin{proof}
By Theorem \ref{tmextension}, it suffices to show that $\mathbb{L}_{\infty,\bif}$ is accumulated by points of $\supp(\mu_\bif)$. Recall that $\Per_1(e^{2i\pi\nu})\cap\mathbb{L}_\infty=\{[1:2\cos(2\pi\nu):0]\}$ for any $0\leq\nu\leq1$ (see Proposition \ref{propMilnor}). Let us fix $0<\theta_0<1$. For $\theta>0$ small enough, the lines $\Per_1(e^{2i\pi\theta_0})$ and $\Per_1(e^{2i\pi(\theta-\theta_0)})$ do not intersect on $\mathbb{L}_\infty$ and therefore,
\begin{center}
$\{\lambda(\theta)\}\pe\Per_1(e^{2i\pi\theta_0})\cap\Per_1(e^{2i\pi(\theta-\theta_0)})\subset\C^2$.
\end{center}
Since $\lambda(\theta)$ has two distinct neutral fixed points, it belongs to $\supp(\mu_\bif)$ (see Theorem \ref{tmsuppmubif}). It remains to check that $\lim_{\theta\rightarrow0}\lambda(\theta)=[1:2\cos(2\pi\theta_0):0]$. By the holomorphic index formula (see \cite{Milnor3}), the multiplier $\gamma(\theta)$ of the third fixed point of $\lambda(\theta)$ is given by
\begin{center}
$\gamma(\theta)=\displaystyle\frac{2-(e^{2i\pi\theta_0}+e^{2i\pi(\theta-\theta_0)})}{1-e^{2i\pi \theta}}$.
\end{center}
As $\lim_{\theta\rightarrow0}|\gamma(\theta)|=+\infty$, one has $\lim_{\theta\rightarrow0}\|\lambda(\theta)\|=+\infty$. The conclusion follows, since $\lambda(\theta)$ stays on $\Per_1(e^{2i\pi\theta_0})$ and $\Per_1(e^{2i\pi\theta_0})\cap\mathbb{L}_\infty=\{[1:2\cos(2\pi\theta_0):0]\}$.
\end{proof}

\par We would like to extend $\mu_\bif$ as a reasonable bifurcation measure on $\p^2$. To this aim, we compare $\mu_\bif$ with $(\hat T_\bif)^2$. We prove the following.

\begin{prop}\label{propComM}
There exists a positive measure $\mu_\infty$ supported by $\mathbb{L}_{\infty,\bif}$ such that
\begin{eqnarray*}
\hat T_\bif\wedge\hat T_\bif= \mu_\bif+\frac{1}{36}\delta_{\infty_{1/2}}+\sum_{1\leq p\leq q/2 \atop p\wedge q=1,q\geq3}\frac{1}{(2^q-1)^2}\delta_{\infty_{p/q}}+\mu_\infty.
\end{eqnarray*}
\end{prop}

\par Les us stress that our previous results ensure the existence of $\mu_\infty$ as a non-negative measure. A recent result due to Kiwi and Rees concerning the number of $(n,m)$-hyperbolic components of $\mathcal{M}_2$ allows to see that $\mu_\infty$ is actually a positive measure.

\begin{proof}
Arguing exactly as in the proof of Lemma \ref{lmintersect}, one sees that $\hat T_\bif\wedge \hat T_\bif$ is a well-defined positive measure on $\p^2$. By definition, $\hat T_\bif\wedge\hat T_\bif$ and $\mu_\bif$ coincide on $\C^2$. To prove the existence of $\mu_\infty$, it thus only remains to justify that
\begin{eqnarray*}
\hat T_\bif\wedge\hat T_\bif\geq\mu_\bif+\frac{1}{36}\delta_{\infty_{1/2}}+\sum_{1\leq p\leq q/2 \atop p\wedge q=1,q\geq3}\frac{1}{(2^q-1)^2}\delta_{\infty_{p/q}}.
\end{eqnarray*}
By Theorem \ref{semicont} and Theorem \ref{tmnombrelelong}, we have
\begin{eqnarray*}
\nu(\hat T_\bif\wedge\hat T_\bif,\infty_{p/q})\geq\nu(\hat T_\bif,\infty_{p/q})^2=\frac{1}{(2^q-1)^2},\text{ when }q\geq3
\end{eqnarray*}
and $\nu(\hat T_\bif\wedge\hat T_\bif,\infty_{1/2})\geq1/36$. The existence of $\mu_\infty$ then follows, since
\begin{center}
$\big( \hat T_\bif\wedge\hat T_\bif\big)\big|_{\mathbb{L}_\infty}\geq\displaystyle\sum_{1\leq p\leq q/2 \atop p\wedge q=1}\nu(\hat T_\bif\wedge\hat T_\bif,\infty_{p/q})\delta_{\infty_{p/q}}$.
\end{center}

~

\par Let us now show that $\mu_\infty>0$. One easily deduce from the convergence of $2^{-n}[\Per_n(0)]$ towards $T_\bif$ in any family (see \cite{BB3}) that
\begin{center}
$\mu_\bif=\lim_m\lim_n2^{-(n+m)}[\Per_n(0)\cap\Per_m(0)]$
\end{center}
on $\C^2$. Thus $\mu_\bif(\C^2)\leq \limsup_m\limsup_n2^{-(n+m)}[\Per_n(0)\cap\Per_m(0)](\C^2)$. Kiwi and Rees proved recently that the number of $(n,m)$-hyperbolic components in $\C^2$ is at most 
\begin{center}
$\displaystyle \left( \frac{5}{48}2^n-\frac{1}{8}\sum_{q=2}^n\frac{\phi(q)\nu_q(n)}{2^q-1}\right)2^m+\mathcal{O}(2^m)$,
\end{center}
where $\nu_q(n)\sim2^{n-1}/(2^q-1)$ and $\nu_q(n)\geq 2^{n-1}/(2^q-1)-1/2$ (see \cite{kiwirees} Theorem $1.1$). A standard transversality statement asserts that this number actually coincides with $[\Per_n(0)\cap\Per_m(0)](\C^2)$ (see \cite{BB3} Theorem 5.2). Thus
\begin{center}
$\displaystyle\mu_\bif(\C^2)\leq\frac{5}{48}-\frac{1}{16}\sum_{q\geq2}\frac{\phi(q)}{(2^q-1)^2}$.
\end{center}
Let us now proceed by contradiction, assuming that $\mu_\infty=0$. Since $\hat T_\bif$ has mass $1/2$, B\'ezout Theorem gives $\|\hat T_\bif\wedge\hat T_\bif\|_{\p^2}=1/4$. Therefore,
\begin{center}
$\displaystyle\frac{1}{4}=\|\mu_\bif\|+\frac{1}{36}+\sum_{1\leq p\leq q/2 \atop p\wedge q=1,q\geq3}\frac{1}{(2^q-1)^2}=\|\mu_\bif\|+\frac{1}{2}\sum_{q\geq2}\frac{\phi(q)}{(2^q-1)^2}-\frac{1}{36}$,
\end{center}
which yields $\frac{25}{56}\leq\sum_{q\geq2}\frac{\phi(q)}{(2^q-1)^2}$. We then have
\begin{eqnarray*}
\frac{25}{56}\leq\sum_{q\geq2}\frac{\phi(q)}{(2^q-1)^2}\leq \frac{\phi(2)}{6}+\frac{\phi(3)}{28}+\frac{1}{8}\sum_{q\geq4}\frac{\phi(q)}{2^q-1}
\end{eqnarray*}
which is impossible, since $\sum_{q\geq1}\frac{\phi(q)}{2^q-1}=2$ (see proof of Lemma \ref{lmmass}), $\phi(2)=1$ and $\phi(3)=2$.
\end{proof}

\begin{rem}
To underline the contrast with the moduli space of cubic polynomials $\mathcal{P}_3$, we recall that the bifurcation measure is compactly supported and coincides with $\hat T_\bif^2$
in $\mathcal{P}_3$.
\end{rem}

\bibliographystyle{short}
\bibliography{biblio}
\bigskip

{\footnotesize 

Fran\c cois Berteloot, Universit\' e Paul Sabatier MIG. 
Institut de Math\'ematiques de Toulouse. 
31062 Toulouse
Cedex 9, France.
{\em Email: berteloo@picard.ups-tlse.fr}}\\

{\footnotesize 

Thomas Gauthier, Universit\'e de Picardie Jules Verne, LAMFA UMR-CNRS 7352, 80039 Amiens Cedex 1, France, and, Stony Brook University, Institute for Mathematical Sciences, Stony Brook, NY 11794,  USA.
{\em Email: thomas.gauthier@u-picardie.fr}}

\end{document}